\documentclass{amsart}

\usepackage{amsmath, amsfonts, amsthm, amssymb}

\usepackage{graphics}
\usepackage{psfrag}

\theoremstyle{plain}
\newtheorem{thm}{Theorem}[section]
\newtheorem{lem}[thm]{Lemma}
\newtheorem{prop}[thm]{Proposition}

\theoremstyle{definition}
\newtheorem{defn}[thm]{Definition}

\newcommand{\N}{\mathbb{N}}
\newcommand{\Z}{\mathbb{Z}}

\newcommand{\bast}{{\displaystyle\ast}}
\newcommand{\mc}[1]{\mathcal{#1}}
\newcommand{\fm}[1]{\mc{#1}^{\pm\bast}}
\newcommand{\hast}{\mbox{\huge $\ast$ \normalsize}\!\!\!}

\DeclareMathOperator{\Area}{Area} 
 \DeclareMathOperator{\Aut}{Aut}
\DeclareMathOperator{\Ab}{Ab}
\DeclareMathOperator{\FreeEq}{\stackrel{\text{free}}{=}}
\DeclareMathOperator{\DEdge}{DEdge}

\begin{document}
\title[A Subgroup of a Direct Product of
Free Groups]{A Subgroup of a Direct Product of Free Groups whose
Dehn Function has a Cubic Lower Bound}
\author{Will Dison}
\address{Will Dison, Department of Mathematics, University Walk, Bristol, BS8 1TW, United Kingdom}
\email{\emph{w.dison@bristol.ac.uk}}
\subjclass[2000]{20F65(primary), 20E05, 20F06 (secondary).}
\date{}

\begin{abstract}
We establish a cubic lower bound on the Dehn function of a certain
finitely presented subgroup of a direct product of $3$ free groups.
\end{abstract}

\maketitle

\begin{center}
\footnotesize
Department of Mathematics\\
University Walk\\
Bristol, BS8 1TW\\
United Kingdom\\
\bigskip
\texttt{w.dison@bristol.ac.uk}
\end{center}

\section{Introduction}

The collection $\mc{S}$ of subgroups of direct products of free
groups is surprisingly rich and has been studied by many authors. In
the early 1960s Stallings \cite{stal63} exhibited a subgroup of
$(F_2)^3$, where $F_2$ is the rank $2$ free group, as the first
known example of a finitely presented group whose third integral
homology group is not finitely generated.  Bieri \cite{Bieri1}
demonstrated that Stallings' group belongs to a sequence of groups
${\rm SB}_n \leq (F_2)^n$, the Stallings-Bieri groups, with ${\rm
SB}_n$ being of type $\mc{F}_{n-1}$ but not of type $\mc{FP}_n$.
(See \cite{Brown1} for definitions and background concerning
finiteness properties of groups.)

In the realm of decision theory, Miha{\u\i}lova \cite{miha58} and
Miller \cite{mill71} exhibited a finitely generated subgroup of
$(F_2)^2$ with undecidable conjugacy and membership problems.  It is
thus seen that even in the $2$-factor case fairly wild behaviour is
encountered amongst the subgroups of direct products of free groups.

In contrast to this, Baumslag and Roseblade \cite{baum84} showed
that in the $2$-factor case this wildness only manifests itself
amongst the subgroups which are not finitely presented.  They proved
that if $G$ is a finitely presented subgroup of $F^{(1)} \times
F^{(2)}$, where $F^{(1)}$ and $F^{(2)}$ are free groups, then $G$ is
itself virtually a direct product of at most $2$ free groups.  This
work was extended by Bridson, Howie, Miller and Short \cite{brid02B}
to an arbitrary number of factors. They proved that if $G$ is a
subgroup of a direct product of $n$ free groups and if $G$ enjoys
the finiteness property $\mc{FP}_n$, then $G$ is virtually a direct
product of at most $n$ free groups. Further information on the
finiteness properties of the groups in class $\mc{S}$ was provided
by Meinert \cite{mein94} who calculated the BNS invariants of direct
products of free groups. These invariants determine the finiteness
properties of all subgroups lying above the commutator subgroup.

Several authors have investigated the isoperimetric behaviour of the
finitely presented groups in $\mc{S}$.  Elder, Riley, Young and this author
\cite{Dison07} have shown that the Dehn function of Stallings' group ${\rm SB}_3$ is quadratic.  The
method espoused by Bridson in \cite{Bridson3} proves that the
function $n^3$ is an upper bound on the Dehn functions of each of
the Stallings-Bieri groups.  In contrast, there have been no results
which give non-trivial lower bounds on the Dehn functions of any
groups in $\mc{S}$.  In other words, until now there has been no
finitely presented subgroup of a direct product of free groups whose
Dehn function was known to be greater than that of the ambient
group. The purpose of this paper is to construct such a subgroup: we
exhibit a finitely presented subgroup of $(F_2)^3$ whose Dehn
function has the function $n^3$ as a lower bound.

Let $K$ be the kernel of a homomorphism $\theta : (F_2)^3
\rightarrow \Z^2$ whose restriction to each factor $F_2$ is
surjective.  By Lemma \ref{lem1} below the isomorphism class of $K$
is independent of the homomorphism chosen.  Results in \cite{mein94}
show that $K$ is finitely presented.

\begin{thm} \label{thm1}
  The Dehn function $\delta$ of $K$ satisfies $\delta(n) \succeq
  n^3$.
\end{thm}

Note that Theorem \ref{thm1} makes no reference to a specific
presentation of $K$ since, as is well known, the Dehn function of a
group is independent (up to $\simeq$-equivalence) of the
presentation chosen.  We refer the reader to Section 2 for
background on Dehn functions, including definitions of the symbols
$\succeq$ and $\simeq$.

The organisation of this paper is as follows.  Section 2 gives basic
definitions concerning Dehn  functions, van Kampen diagrams and
Cayley complexes.  We expect that the reader will already be
familiar with these concepts; the purpose of the exposition is
principally to introduce notation.  In Section 3 we define a class
of subgroups of direct products of free groups of which $K$ is a
member.  The subsequent section gives finite generating sets for
those groups in the class which are finitely generated and Section 5
proves a splitting theorem which shows how certain groups in the class
decompose as amalgamated products.  We then prove in Section 6 how,
in certain circumstances, the distortion of the subgroup $H$ in an
amalgamated product $\Gamma = G_1 \ast_H G_2$ gives rise to a lower
bound on the Dehn function of $\Gamma$. Theorem \ref{thm1} follows
as a corollary of this result when applied to the splitting of $K$
given in Section 5.

\section{Dehn Functions} \label{sec1}

In this section we recall the basic definitions concerning Dehn
functions of finitely presented groups.  All of this material is standard.
For further background and
a more thorough exposition, including proofs, see, for example,
\cite{brid02} or \cite{Riley1}.

Given a set $\mc{A}$, write $\mc{A}^{-1}$ for the set of formal
inverses of the elements of $\mc{A}$ and write $\mc{A}^{\pm1}$ for
the set $\mc{A} \cup \mc{A}^{-1}$.  Denote by $\mc{A}^{\pm\bast}$
the free monoid on the set $\mc{A}^{\pm1}$.  We refer to elements of
$\mc{A}^{\pm\bast}$ as words in the letters $\mc{A}^{\pm1}$ and
write $|w|$ for the length of such a word $w$. Given words $w_1, w_2
\in \fm{A}$ we write $w_1 \FreeEq w_2$ if $w_1$ and $w_2$ are freely
equal and $w_1 \equiv w_2$ if $w_1$ and $w_2$ are equal as elements
of $\fm{A}$.

\begin{defn} \label{def1}
  Let $\mc{P} = \langle \mc{A} \, | \mc{R} \rangle$ be a finite presentation of a group $G$.
  A word $w \in \mc{A}^{\pm\bast}$ is said to be
  \emph{null-homotopic} over $\mc{P}$ if it represents the
  identity in $G$.  A \emph{null-$\mc{P}$-expression} for such a
  word is a sequence $(x_i, r_i)_{i=1}^m$ in $\fm{A} \times
  \mc{R}^{\pm1}$ such that \[w \FreeEq \prod_{i=1}^m x_i r_i
  x_i^{-1}.\]  Define the \emph{area} of a null-$\mc{P}$-expression to be the integer $m$
  and define the \emph{$\mc{P}$-area} of $w$,
  written $\Area_\mc{P}(w)$, to be the minimal area taken over all
  null-$\mc{P}$-expressions for $w$.

  The Dehn function of the presentation $\mc{P}$, written $\delta_\mc{P}$,
  is defined to be the function $\N \rightarrow \N$ given by \[\delta_\mc{P}(n)
  = \max \{ \Area_\mc{P}(w) \, : \, \text{$w \in \fm{A}$, $w$ null-homotopic,
  $|w| \leq n$} \}.\]
\end{defn}

Although the Dehn functions of different finite presentations of a
fixed group may differ, their asymptotic behaviour will be the same.
This is made precise in the following way.

\begin{defn}
  Given functions $f, g : \N \rightarrow \N$, write $f \preceq g$ if
  there exists a constant $C > 0$ such that $f(n) \leq Cg(Cn + C) +
  Cn + C$ for all $n$.  Write $f \simeq g$ if $f \preceq g$ and
  $g \preceq f$.
\end{defn}

\begin{lem}
  If $\mc{P}_1$ and $\mc{P}_2$ are finite presentations of the
  same group then $\delta_{\mc{P}_1} \simeq \delta_{\mc{P}_2}$.
\end{lem}

For a proof of this standard result see, for example, \cite[Proposition~1.3.3]{brid02}.

A useful tool for the study of Dehn functions is
a class of objects known as van Kampen diagrams.
Roughly speaking, these are planar CW-complexes which portray
diagrammatically schemes for reducing null-homotopic words to the
identity.  Such diagrams, whose definition is recalled below, allow the application of topological
methods to the calculation of Dehn functions. For background and further details see, for example, \cite[Section~4]{brid02}.  For the definition of a combinatorial CW-complex see, for example, \cite[Appendix~A]{brid02}.

\begin{defn}
  A \emph{singular disc diagram} $\Delta$ is a finite, planar,
  contractible combinatorial CW-complex with a specified base vertex
  $\star$ in its boundary.  The \emph{area} of $\Delta$, written
  $\Area(\Delta)$, is defined to be the number of $2$-cells of which
  $\Delta$ is composed.  The \emph{boundary cycle} of $\Delta$ is
  the edge loop in $\Delta$ which starts at $\star$ and traverses
  $\partial \Delta$ in the anticlockwise direction.  The
  interior of $\Delta$ consists of a number of disjoint open
  $2$-discs, the closures of which are called the \emph{disc components} of $\Delta$.

  Each $1$-cell of $\Delta$ has associated to it two directed
  edges $\epsilon_1$ and $\epsilon_2$, with $\epsilon_1^{-1} =
  \epsilon_2$.  Let $\DEdge(\Delta)$ be the set of directed edges
  of $\Delta$.  A \emph{labelling} of $\Delta$ over a set $\mc{A}$ is
  a map $\lambda : \DEdge(\Delta) \rightarrow \mc{A}^{\pm1}$ such that
  $\lambda(\epsilon^{-1}) = \lambda(\epsilon)^{-1}$.  This induces a
  map from the set of edge paths in $\Delta$ to $\fm{A}$.
  The \emph{boundary label} of $\Delta$ is the word in $\fm{A}$
  associated to the boundary cycle.

  Let $\mc{P} = \langle \mc{A} \, | \, \mc{R} \rangle$ be a finite
  presentation.  A \emph{$\mc{P}$-van Kampen diagram} for a word $w
  \in \fm{A}$ is a singular disc diagram $\Delta$ labelled over $\mc{A}$ with
  boundary label $w$ and such that for each $2$-cell $c$ of $\Delta$
  the anticlockwise edge loop given by the attaching map of $c$,
  starting at some vertex in $\partial c$, is labelled by a word in
  $\mc{R}^{\pm1}$.
\end{defn}

\begin{lem}[Van Kampen's Lemma]
  A word $w \in \fm{A}$ is null-homotopic over $\mc{P}$ if and only if there
  exists a $\mc{P}$-van Kampen diagram for $w$.  In this case the
  $\mc{P}$-area of $w$ is the minimal area over all $\mc{P}$-van
  Kampen diagrams for $w$.
\end{lem}

For a proof of this result see, for example, \cite[Theorem~4.2.2]{brid02}.

Associated to a presentation $\mc{P} = \langle \mc{A} \, | \, \mc{R}
\rangle$ of a group $G$ there is a standard combinatorial
$2$-complex $K_\mc{P}$ with $\pi_1(K_\mc{P}) \cong G$.  The complex
$K_\mc{P}$ is constructed by taking a wedge of copies of $S^1$
indexed by the letters in $\mc{A}$ and attaching $2$-cells indexed
by the relations in $\mc{R}$.  The $2$-cell corresponding to a
relation $r \in \mc{R}$ has $|r|$ edges and is attached by
identifying its boundary circuit with the edge path in $K_\mc{P}^1$
along which one reads the word $r$.

The Cayley $2$-complex associated to $\mc{P}$, denoted by
$Cay^2(\mc{P})$, is defined to be the universal cover of $K_\mc{P}$.
If one chooses a base vertex of $Cay^2(\mc{P})$ to represent the
identity element of $G$ then the $1$-skeleton of this complex is
canonically identified with the Cayley graph of $\mc{P}$.  Given a
$\mc{P}$-van Kampen diagram $\Delta$ there is a unique
label-preserving combinatorial map from $\Delta$ to $Cay^2(\mc{P})$
which maps the base vertex of $\Delta$ to the vertex of
$Cay^2(\mc{P})$ representing the identity.

\section{A Class of Subgroups of Direct Products of Free Groups}

In this section we introduce a class of subgroups of direct products
of free groups of which the group $K$ defined in the introduction
will be a member.  We first fix some notation which will be used
throughout the paper. Given integers $i, m \in \N$, let $F^{(i)}_m$
be the rank $m$ free group with basis $e^{(i)}_1, \ldots,
e^{(i)}_m$.  Given an integer $r \in \N$, let $\Z^r$ be the rank $r$
free abelian group with basis $t_1, \ldots, t_r$.

Given positive integers $n, m \geq 1$ and $r \leq m$, we wish to
define a group $K^n_m(r)$ to be the kernel of a homomorphism
$\theta: F^{(1)}_m \times \ldots \times F^{(n)}_m \rightarrow \Z^r$
whose restriction to each factor $F^{(i)}_m$ is surjective.  For
fixed $n$, $m$ and $r$, the isomorphism class of the group
$K^n_m(r)$ is, up to an automorphism of the factors of the ambient
group $F^{(1)}_m \times \ldots \times F^{(n)}_m$, independent of the
homomorphism $\theta$. This is proved by the following lemma.

\begin{lem} \label{lem1}
Let $F$ be a rank $m$ free group.  Given a surjective homomorphism
$\phi : F \rightarrow \Z^r$, there exists a basis $e_1, \ldots, e_m$
of $F$ so that
\begin{equation*} \phi(e_i) = \begin{cases} t_i & \text{if $1 \leq
i \leq r$,} \\ 0 & \text{if $r+1 \leq i \leq m$.} \end{cases}
\end{equation*}
\end{lem}

\begin{proof}
The homomorphism $\phi$ factors through the abelianisation homomorphism $\Ab : F
\rightarrow A$, where $A$ is the rank $m$ free abelian group $F /
[F, F]$, as $\phi = \bar{\phi} \circ \Ab$ for some homomorphism
$\bar{\phi} : A \rightarrow \Z^r$. Since $\bar{\phi}$ is surjective,
$A$ splits as $A_1 \oplus A_2$ where $\bar{\phi}$ is an isomorphism
on the first factor and $0$ on the second factor. There thus exists
a basis $s_1, \ldots, s_m$ for $A$ so as
\begin{equation*} \bar{\phi}(s_i) = \begin{cases} t_i & \text{if
$1 \leq i \leq r$,} \\ 0 & \text{if $r+1 \leq i \leq m$.}
\end{cases} \end{equation*}

We claim that the $s_i$ lift under $\Ab$ to a basis for $F$.  To see
this let $f_1, \ldots , f_m$ be any basis for $F$ and let
$\bar{f}_1, \ldots, \bar{f}_m$ be its image under $\Ab$, a basis for
$A$.  Let $\rho \in \Aut(A)$ be the change of basis isomorphism from
$\bar{f}_1, \ldots, \bar{f}_m$ to $s_1, \ldots, s_m$.  It suffices
to show that this lifts under $\Ab$ to an automorphism of $F$.  But
this is certainly the case since $\Aut(A) \cong GL_m(\Z)$ is
generated by the elementary transformations and each of these
obviously lifts to an automorphism.
\end{proof}

\begin{defn} \label{def2}
For integers $n, m \geq 1$ and $r \leq m$, define $K^n_m(r)$ to be
the kernel of the homomorphism $\theta: F^{(1)}_m \times \ldots
\times F^{(n)}_m \rightarrow \Z^r$ given by
\begin{equation*} \theta(e^{(i)}_j) =
\begin{cases} t_j & \text{if $1 \leq j \leq r$,} \\ 0 & \text{if $r+1
\leq j \leq m$.}
\end{cases} \end{equation*}
\end{defn}

Note that $K^n_2(1)$ is the $n^\text{th}$ Stallings-Bieri group
${\rm SB}_n$.

By a result in Section 1.6 of \cite{mein94}, if $r \geq 1$ and $m
\geq 2$ then $K^n_m(r)$ is of type $\mc{F}_{n-1}$ but not of type
$\mc{FP}_{n}$. In particular the group $K \cong K^3_2(2)$ defined in
the introduction is finitely presented.

\section{Generating Sets}

We give finite generating sets for those groups $K^n_m(r)$ which are finitely generated.  We make use of the following notational
shorthand: given formal symbols $x$ and $y$, write $[x, y]$ for $x y x^{-1} y^{-1}$ and $x^y$ for $y x y^{-1}$.

\begin{prop} \label{prop1}
If $n \geq 2$ then $K^n_m(r)$ is generated by $S_1
\cup S_2 \cup S_3$ where \begin{align*} S_1 &=
\{e_i^{(1)} \big(e_i^{(j)}\big)^{-1} \, : \, 1 \leq i \leq r , 2 \leq j \leq
n\}, \\ S_2 &= \{e_i^{(j)} \, : \, r+1 \leq i \leq m, 1 \leq j \leq
n\}, \\ S_3 &= \{\big[e_i^{(1)}, e_j^{(1)}\big] \, : \, 1 \leq i < j \leq r
\}. \end{align*}
\end{prop}

\begin{proof}
Partition $S_2$ as $S_2' \cup S_2''$ where $$S_2' = \{e_i^{(j)} \, :
\, r+1 \leq i \leq m, 2 \leq j \leq n\}$$ and $$S_2'' = \{e_i^{(1)}
\, : \, r+1 \leq i \leq m\}.$$  Project $K^n_m(r) \leq F^{(1)}_m
\times \ldots \times F^{(n)}_m$ onto the last $n-1$ factors to give
the short exact sequence $1 \rightarrow K^1_m(r) \rightarrow K_m^n(r)
\rightarrow F^{(2)}_m \times \ldots \times F^{(n)}_m \rightarrow 1$.
Note that $S_1 \cup S_2'$ projects to a set of generators for
$F^{(2)}_m \times \ldots \times F^{(n)}_m$ and that $K^1_m(r)$ is
the normal closure in $F^{(1)}_m$ of $S_2'' \cup S_3$. If $\zeta \in
F^{(1)}_m$ and $w \equiv w(e_1^{(1)}, \ldots , e_m^{(1)})$ is a word
in the generators of $F^{(1)}_m$ then
$$\zeta^w = \zeta ^ {w\left(e_1^{(1)}(e_1^{(2)})^{-1}, \ldots,
e_m^{(1)}(e_m^{(2)})^{-1}\right)}.$$ Thus $S_1 \cup S_2' \cup S_2''
\cup S_3$ generates $K^n_m(r)$.
\end{proof}

\section{A Splitting Theorem}

The following result gives an amalgamated product decomposition of the groups $K^n_m(r)$ in the case that $r=m$.  Note that with slightly more work one could prove a more general result without this restriction.  We introduce the following notation: given a collection of groups $M, L_1, \ldots, L_k$ with $M \leq L_i$
for each $i$, we denote by $\hast_{i=1}^k(L_i \, ; \, M)$ the
amalgamated product $L_1 \ast_M \ldots \ast_M L_k$.

\begin{thm} \label{thm2}
If $n \geq 2$ and $m \geq 1$ then $K^n_m(m) \cong \hast_{k=1}^m(L_k \, ; \, M)$, where $M = K^{n-1}_m(m)$ and, for each $k=1, \ldots, m$, the group $L_k \cong K^{n-1}_m(m-1)$ is the kernel of the
homomorphism $$\theta_k : F^{(1)}_m \times \ldots \times F^{(n-1)}_m
\rightarrow \Z^{m-1}$$ given by
\begin{equation*} \theta_k(e_j^{(i)}) =
\begin{cases} t_j & \text{if $1 \leq j \leq k-1$,} \\
0  & \text{if $j = k$,} \\ t_{j-1} & \text{if $k+1 \leq j \leq m$.}
\end{cases}
\end{equation*}
\end{thm}

\begin{proof}
Projecting $K^n_m(m)$ onto the factor $F^{(n)}_m$ gives the short
exact sequence $1 \rightarrow K^{n-1}_m(m) \rightarrow K^n_m(m)
\rightarrow F^{(n)}_m \rightarrow 1$.  This splits to show that
$K^n_m(m)$ has the structure of an internal semidirect product $M
\rtimes \hat{F}^{(n)}_m$ where $\hat{F}^{(n)}_m \cong F^{(n)}_m$ is
the subgroup of $F^{(n-1)}_m \times F^{(n)}_m$ generated by
$$e_1^{(n-1)}\big(e_1^{(n)}\big)^{-1}, \, \ldots, \,
e_m^{(n-1)}\big(e_m^{(n)}\big)^{-1}.$$ Since the action by conjugation of
$e_k^{(n-1)}\big(e_k^{(n)}\big)^{-1}$ on $M$ is the same as the action of
$e_k^{(n-1)}$, we have that
\begin{align*}K^n_m(m) &= M \rtimes
\hat{F}^{(n)}_m \\
&\cong \hast_{k=1}^m \Big(M \rtimes \Big\langle
e_k^{(n-1)}\big(e_k^{(n)}\big)^{-1} \Big\rangle \, ; \, M \Big) \\
&\cong \hast_{k=1}^m \Big(M \rtimes \Big\langle e_k^{(n-1)}
\Big\rangle \, ; \, M \Big).
\end{align*}

Define a homomorphism $p_k : F^{(1)}_m \times \ldots \times
F^{(n-1)}_m \rightarrow \Z$ by
\begin{equation*} p_k\big(e_j^{(i)}\big) = \begin{cases} 1 & \text{if
$j=k$,} \\ 0 & \text{otherwise,} \end{cases} \end{equation*} and
note that $L_k \cap \ker p_k$ is the kernel $K^{n-1}_m(m)$ of the standard
homomorphism $\theta : F^{(1)}_m \times \ldots \times F^{(n-1)}_m
\rightarrow \Z^m$ given in Definition \ref{def2}. Considering the
restriction of $p_k$ to $L_k$ gives the short exact sequence $1
\rightarrow K^{n-1}_m(m) \rightarrow L_k \rightarrow \Z \rightarrow
1$ which demonstrates that $L_k = K^{n-1}_m(m) \rtimes \langle
e_k^{(n-1)} \rangle$.
\end{proof}

\section{Dehn Functions of Amalgamated Products}

In this section we will be concerned with finitely presented
amalgamated products $\Gamma = G_1 \ast_H G_2$ where $H$, $G_1$ and $G_2$ are finitely generated groups and $H$ is proper in each $G_i$.  Suppose each $G_i$ is presented by
$\langle\mathcal{A}_i \, | \, \mathcal{R}_i \rangle$, with
$\mathcal{A}_i$ finite.  Note that we are at liberty to choose the
$\mc{A}_i$ so as each $a \in \mc{A}_i$ represents an element of $G_i
\smallsetminus H$.  Indeed, since $H$ is proper in $G_i$ there
exists some $a' \in \mc{A}_i$ representing an element of $G_i
\smallsetminus H$ and we can replace each other element $a \in
\mc{A}_i$ by $a' a$ if necessary.

Let $\mathcal{B}$ be a finite generating set for $H$ and for each $b
\in \mathcal{B}$ choose words $u_b \in \mathcal{A}_1^{\pm\bast}$ and
$v_b \in \mathcal{A}_2^{\pm\bast}$ which equal $b$ in $\Gamma$.
Define $\mc{E} \subset (\mc{A}_1 \cup \mc{A}_2 \cup
\mc{B})^{\pm\bast}$ to be the finite collection of words
$\{bu_b^{-1}, bv_b^{-1} \, : \, b \in \mc{B}\}$. Then, since
$\Gamma$ is finitely presented, there exist finite subsets
$\mathcal{R}_1' \subseteq \mathcal{R}_1$ and $\mathcal{R}_2'
\subseteq \mathcal{R}_2$ such that $\Gamma$ is finitely presented by
$$\mathcal{P} = \langle \mathcal{A}_1, \mathcal{A}_2,
\mathcal{B} \, | \, \mathcal{R}_1', \mathcal{R}_2', \mc{E} \rangle.
$$

\begin{thm} \label{thm3}
Let $w \in \mc{A}_1^{\pm\bast}$ be a word representing an element $h
\in H$ and let $u \in \mathcal{A}_1^{\pm\bast}$ and $v \in
\mathcal{A}_2^{\pm\bast}$ be words representing elements $\alpha \in
G_1 \smallsetminus H$ and $\beta \in G_2 \smallsetminus H$ respectively.  If
$[\alpha, h] = [\beta , h] = 1$ then
$$\Area_\mathcal{P}([w, (uv)^n]) \geq 2n \, {\rm d}_\mathcal{B}(1, h)$$ where ${\rm d}_\mathcal{B}$ is
the word metric on $H$ associated to the generating set
$\mathcal{B}$.
\end{thm}

\begin{proof}
Let $\Delta$ be a $\mc{P}$-van Kampen diagram for the null-homotopic
word $[w, (uv)^n]$ (see Diagram~\ref{fig1}).
\begin{figure}[htbp]
  \psfrag{u}{$u$}
  \psfrag{v}{$v$}
  \psfrag{p1}{$p_1$}
  \psfrag{p2}{$p_2$}
  \psfrag{pn}{$p_n$}
  \psfrag{q1}{$q_1$}
  \psfrag{q2}{$q_2$}
  \psfrag{qn}{$q_n$}
  \psfrag{w}{$w$}
  \centering \includegraphics{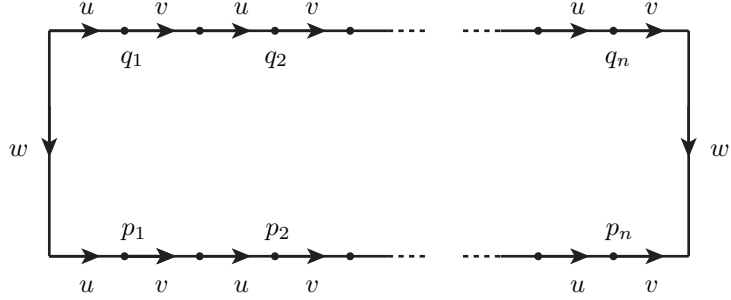}
  \caption{The van Kampen diagram $\Delta$} \label{fig1}
\end{figure}
For each $i = 1 , 2, \ldots, n$, define $p_i$ to
be the vertex in $\partial \Delta$ such that the anticlockwise path
in $\partial \Delta$ from the basepoint around to $p_i$ is labelled
by the word $w(uv)^{i-1}u$.  Similarly define $q_i$ to be the vertex
in $\partial \Delta$ such that the anticlockwise path in $\partial
\Delta$ from the basepoint around to $q_i$ is labelled by the word
$w (uv)^n w^{-1} (uv)^{i-n} v^{-1}$.  We will show that for each $i$
there is a $\mathcal{B}$-path (i.e. an edge path in $\Delta$
labelled by a word in the letters $\mathcal{B}$) from $p_i$ to
$q_i$.

We assume that the reader is familiar with Bass-Serre theory, as
exposited in \cite{Serre1}.  Let $T$ be the Bass-Serre tree
associated to the splitting $G_1 \ast_H G_2$. This consists of an
edge $gH$ for each coset $\Gamma / H$ and a vertex $gG_i$ for each
coset $\Gamma / G_i$.  The edge $gH$ has initial vertex $gG_1$ and
terminal vertex $gG_2$.  We will construct a continuous (but
non-combinatorial) map $\Delta \rightarrow T$ as the composition of
the natural map $\Delta \rightarrow Cay^2(\mc{P})$ with the map $f :
Cay^2(\mc{P}) \rightarrow T$ defined below.

There is a natural left action of $\Gamma$ on each of
$Cay^2(\mc{P})$ and $T$ and we construct $f$ to be equivariant with
respect to this as follows.  Let $m$ be the midpoint of the edge $H$
of $T$ and define $f$ to map the vertex $g \in Cay^2(\mc{P})$ to the
point $g \cdot m$, the midpoint of the edge $gH$.  Define $f$ to map
the edge of $Cay^2(\mc{P})$ labelled $a \in \mathcal{A}_i$ joining
vertices $g$ and $ga$ to the geodesic segment joining $g \cdot m$ to
$ga \cdot m$. Since $a \not\in H$ this segment is an embedded arc of
length $1$ whose midpoint is the vertex $gG_i$.  Define $f$ to
collapse the edge in $Cay^2(\mc{P})$ labelled $b \in \mathcal{B}$
joining vertices $g$ and $gb$ to the point $g \cdot m = gb \cdot m$.
This is well defined since $gH = gbH$. This completes the definition
of $f$ on the $1$-skeleton of $\Delta$; we now extend $f$ over the
$2$-skeleton.

Let $c$ be a $2$-cell in $Cay^2(\mc{P})$ and let $g$ be some vertex
in its boundary.  Assume that $c$ is metrised so as to be convex and
let $l$ be some point in its interior.  The form of the relations in
$\mc{P}$ ensures that the boundary label of $c$ is a word in the
letters $\mathcal{A}_i \cup \mathcal{B}$ for some $i$ and so every
vertex in $\partial c$ is labelled $gg'$ for some $g' \in G_i$. Thus
$f$ as so far defined maps $\partial c$ into the ball of radius
$1/2$ centred on the vertex $g G_i$; we extend $f$ to the interior
of $c$ by defining it to map the geodesic segment $[l, p]$, where $p
\in \partial c$, to the geodesic segment $[gG_i, f(p)]$.  This is
independent of the vertex $g \in \partial c$ chosen and makes $f$
continuous since geodesics in a tree vary continuously with their
endpoints.  We now define $\bar{f} : \Delta \rightarrow T$ to be the
map given by composing $f$ with the label-preserving map $\Delta
\rightarrow Cay^2(\mc{P})$ which sends the basepoint of $\Delta$ to
the vertex $1 \in Cay^2(\mc{P})$.

Since $w$ commutes with $u$ and $v$ we have that $\bar{f}(p_i) =
\mbox{$w(uv)^{i-1}u \cdot m$} = \mbox{$(uv)^{i-1} u \cdot m$} =
\bar{f} (q_i)$; define $S$ to be the preimage under $\bar f$ of this
point. By construction, the image of the interior of each $2$-cell
in $\Delta$ and the image of the interior of each
$\mathcal{A}_i$-edge is disjoint from $\bar f(p_i)$.  Thus $S$
consists of vertices and $\mathcal{B}$-edges and so finding a
$\mathcal{B}$-path from $p_i$ to $q_i$ reduces to finding a path in
$S$ connecting these vertices. Let $s_i$ and $t_i$ be the vertices
of $\partial \Delta$ immediately preceding and succeeding $p_i$ in
the boundary cycle.  Unless $h=1$, in which case the theorem is
trivial, the form of the word $[w, (uv)^n]$, together with the
normal form theorem for amalgamated products, implies that all the
vertices $p_i$, $s_i$ and $t_i$ lie in the boundary of the same disc
component $D$ of $\Delta$. Furthermore, since $u$ and $v$ are words
in the letters $\mathcal{A}_1$ and $\mathcal{A}_2$ respectively, the
points $f(s_i)$ and $f(t_i)$ are separated in $T$ by $f(p_i)$. Thus
$s_i$ and $t_i$ are separated in $D$ by $S$ and so there exists an
edge path $\gamma_i$ in $S$ from $p_i$ to some other vertex $r_i \in
\partial D$. Since $\gamma_i$ is a $\mathcal{B}$-path it follows
that the word labelling the sub-arc of the boundary cycle of
$\Delta$ from $p_i$ to $r_i$ represents an element of $H$, and, by
considering subwords of $[w, (uv)^n]$, we see that the only
possibility is that $r_i = q_i$. Thus, for each $i = 1, \ldots, n$,
the path $\gamma_i$ gives the required $\mathcal{B}$-path connecting
$p_i$ to $q_i$.  We choose each $\gamma_i$ to contain no repeated
edges.

For $i \neq j$, the two paths $\gamma_i$ and $\gamma_j$ are disjoint,
since if they intersected there would be a $\mathcal{B}$-path
joining $p_i$ to $p_j$ and thus the word labelling the subarc of the
boundary cycle from $p_i$ to $p_j$ would represent an element of
$H$. Observe that no two edges in any of the paths $\gamma_1, \ldots
, \gamma_n$ lie in the boundary of the same $2$-cell in $\Delta$
since each relation in $\mathcal{P}$ contains at most one occurrence
of a letter in $\mathcal{B}$. Because the word labelling $\partial
\Delta$ contains no occurrences of a letter in $\mathcal{B}$ the
interior of each edge of a path $\gamma_i$ lies in the interior of
$\Delta$ and thus in the boundary of two distinct $2$-cells.  Since
each path $\gamma_i$ contains no repeated edges we therefore obtain
the bound $\Area(\Delta) \geq \sum_{i=1}^n 2|\gamma_i|$. But the
word labelling each $\gamma_i$ is equal to $h$ in $\Gamma$ and so
the length of $\gamma_i$ is at least ${\rm d}_\mathcal{B}(1, h)$
whence we obtain the required inequality.
\end{proof}

We are now in a position to prove the main theorem.

\begin{proof}[Proof of Theorem \ref{thm1}]
To avoid excessive superscripts we change notation and write $x_i,
y_i$ for the generators of $F^{(i)}_2$, $i=1,2$.

By Proposition \ref{prop1} and Theorem \ref{thm2} we have that $K
\cong K^3_2(2) \cong L_1 \ast_M L_2$ where, as subgroups of
$F_2^{(1)} \times F_2^{(2)}$, $L_1 = K^2_2(1)$ is generated by
$\mathcal{A}_1 = \{x_1x_2^{-1}, y_1, y_2\}$, $L_2 \cong K^2_2(1)$ is
generated by $\mathcal{A}_2 = \{x_1, x_2, y_1y_2^{-1}\}$ and $M =
K^2_2(2)$ is generated by $\mathcal{B} = \{x_1x_2^{-1}, y_1y_2^{-1},
[x_1, y_1] \}$.  To obtain the generating set for $L_2$ we have here
implicitly used the automorphism of $F_2^{(1)} \times F_2^{(2)}$
which interchanges $x_i$ with $y_i$ and realises the isomorphism
between $L_2$ and $K_2^2(1)$.

For $n \in \N$, define $h_n$ to be the element $[x_1^n, y_1^n] \in
K^2_2(2)$ and define $w_n$ to be the word $[(x_1x_2^{-1})^n, y_1^n]
\in \mc{A}_1^{\pm\bast}$ representing $h_n$.  Note that $h_n$
commutes with both $y_2 \in \mathcal{A}_1$ and $x_2 \in
\mathcal{A}_2$ so by Theorem \ref{thm3} the word $[w_n,
(y_2x_2)^n]$, which has length $16n$, has area at least $2n \,
d_\mathcal{B}(1, h_n)$.  We claim that $d_\mathcal{B}(1, h_n) \geq
n^2$.

Suppose that in $F_2^{(1)} \times F_2^{(2)}$ the element $h_n$ is
represented by a word $w \equiv w(x_1x_2^{-1}, y_1y_2^{-1}, [x_1,
y_1])$ in the generators $\mathcal{B}$.  Let $k$ be the number of
occurrences of the third variable in the word $w$.  We will show
that $k \geq n^2$.

Observe that, as a group element, the word $w(x_1x_2^{-1}, y_1y_2^{-1},
[x_1, y_1])$ is equal to the word $w(x_1, y_1, [x_1, y_1]) \,
w(x_2^{-1}, y_2^{-1}, 1).$ Thus we have that $[x_1^n, y_1^n]$ is
freely equal to $w(x_1, y_1, [x_1, y_1])$ and that $w(x_2^{-1},
y_2^{-1}, 1)$, and thus $w(x_1, y_1, 1)$, is freely equal to the
empty word. It follows that $[x_1^n, y_1^n]$ can be converted to the
empty word by free expansions, free contractions and deletion of $k$
subwords $[x_1, y_1]$. Hence $[x_1^n, y_1^n]$ is a null-homotopic
word over the presentation $\mc{P} = \langle x_1, y_1 \, | \, [x_1,
y_1] \rangle$ with $\mc{P}$-area at most $k$. But $\mc{P}$ presents the rank $2$ free abelian
group, and basic results on Dehn functions give that $[x_1^n,
y_1^n]$ has area $n^2$ over this presentation. Thus $k \geq n^2$.
\end{proof}

Note that the above proof also shows that $K^2_2(2)$ has at least
quadratic distortion in each of $K^2_2(1)$ and $F_2^{(1)} \times
F_2^{(2)}$.  In fact it can be shown that the distortion is
precisely quadratic.

\bigskip \emph{Acknowledgements.}  I would like to thank my thesis
advisor, Martin Bridson, for his many helpful comments made during
the preparation of this article.

\end{document}